\documentclass{article}[12pt]
\usepackage[a4paper, total={6.5in, 9in}]{geometry}
\usepackage{amssymb}
\usepackage{amsthm}
\usepackage{amsmath}
\usepackage{enumitem}
\usepackage{hyperref}
\usepackage{authblk}
\usepackage{mathtools}

\newtheorem{theorem}{\large Theorem}[section]
\newtheorem{lemma}{\large Lemma}[section]

\newtheorem{definition}{\large Definition}[section]
\newtheorem{remark}{\large Remark}[section]
\newtheorem{example}{\large Example}[section]

\def\1{\rule{0pt}{1.7ex}xy}

\def\2{\rule{0pt}{1.7ex}x_0x}
\def\3{\rule{0pt}{2ex}X_x}

\def\4{\rule{0pt}{1.7ex}1}
\def\5{\rule{0pt}{1.7ex}2}

\begin{document}

	\thispagestyle{empty}
	
	\title{Generalized Hukuhara directional differentiability of interval-valued functions on Riemannian manifolds}
	
	\author[a]{Hilal Ahmad Bhat}

	\author[b,*]{Akhlad Iqbal}

\affil[a]{\it \small Department of Mathematics, Aligarh Muslim University, Aligarh, 202002, Uttar Pradesh, India}

\affil[b]{\it \small Department of Mathematics, Aligarh Muslim University, Aligarh, 202002, Uttar Pradesh, India}
\date{}

\maketitle

\let\thefootnote\relax\footnotetext{*corresponding author (Akhlad Iqbal)\\ Email addresses: bhathilal01@gmail.com (Hilal Ahmad Bhat),\\ akhlad6star@gmail.com (Akhlad Iqbal)}


\begin{abstract}
	In this paper, we show that generalized Hukuhara directional differentiability of an interval-valued function (IVF) defined on Riemannian manifolds is not equivalent to the directional differentiability of its center and half-width functions and hence not to its end point functions. This contrasts with S.-L. Chen's \cite{chen} assertion which says the equivalence holds in terms of endpoint functions of an IVF which is defined on a Hadamard manifold. Additionally, the paper addresses some other inaccuracies which arise when assuming the convexity of a function at a single point in its domain. In light of these arguments, the paper presents some basic results that relate to both the convexity and directional differentiability of an IVF.
\end{abstract}

\section{Introduction}
The Karush-Kuhn-Tucker (KKT) conditions play a significant role in optimization programming problems. Kuhn and Tucker (1951) \cite{kuhn} originally formulated the conditions, which were subsequently referred to as the Kuhn-Tucker conditions. However, it was later discovered that Karush (1939) \cite{karush} had already presented these conditions in his Master's thesis. As a result, the conditions became known as the Karush-Kuhn-Tucker (KKT) conditions. For a last few decades, the authors have been developing these conditions for different areas of mathematics such as stochastic optimization programming (SOP), fuzzy optimization programming (FOP) and interval-valued optimization programming (IVOP). For the case of IVOP, one can refer to \cite{chalco-cano,chen,sadikur,bilal1,bilal2,wu1,wu2}. 
 
 Moreover, while developing KKT conditions for an IVOP problem, the authors \cite{chalco-cano,chen,bilal1,bilal2,wu1,wu2} have assumed the convexity of objective function and constraints mainly at a point. However, the properties enjoyed by a function which is convex at each point of its domain (convex) may not be inherited by a function which is convex only at a single point of its domain[\cite{bazara}, Page-145]. Nevertheless, Remark 2.1(i)*\let\thefootnote\relax\footnotetext{*Remark 2.1(i) in \cite{chen} is as following:\\ If a function $f:D \rightarrow \mathbb{R}$, defined on a nonempty open geodesic convex subset of a Hadamard manifold M, is convex at $x\in D$, then $f$ is directionally differentiable at $x$.} in \cite{chen} asserts that if a real-valued function defined on a nonempty open geodesic convex subset of a Hadamard manifold is convex at a point, then the function is necessarily directionally differentiable there at. However, the Remark is false and the counter example is given in Appendix A. Furthermore, Chen \cite{chen} claims that the gH directional differentiability at a point of an IVF defined on Hadamard manifold is equivalent to the directional differentiability of its endpoint functions. The claim lies in $\text{Lemma 3.2}^{\dagger}$\let\thefootnote\relax\footnotetext{$\dagger$Lemma 3.2 in \cite{chen} is as following:\\A function $f:D \rightarrow \mathbb{I}$, with $f(x) = [f^L(x), f^U(x)]$, is gH-directional differentiable at $x$ in the direction $v \in T_x(M)$ if and only if $f^L$ and $f^U$ are directional differentiable at $x$ in the direction of $v$.} in \cite{chen}. But, the Lemma is not necessarily true in general and the counter example is given in Appendix A. It is worth noting that these imprecise assertions i.e., Remark 2.1(i) and Lemma 3.2 in \cite{chen},  are frequently used in deducing some major conclusions in \cite{chen}.

In view of the above discussion, we show that generalized Hukuhara directional differentiability (gH-directional differentiability) of an IVF defined on Riemannian manifold is not equivalent to the directional differentiability of its center and half-width functions. Although, we show the equivalence holds for a particular class of functions. Furthermore, we study few basic properties of interval-valued convex function on Riemannian manifolds that are related to gH-directional differentiability. The focus is laid on the convexity of a function at a point. Throughout the paper, we will assume (M,g) as complete finite dimensional Riemannian manifold with Riemannian metric $g$ and Riemannian connection $\nabla$ on $M$.

\section{Preliminaries}
In this section, we recall the basic arithmetics of intervals.

Let $\mathbb{I}$ be the set of all closed and bounded intervals of $\mathbb{R}$. Let $A\in \mathbb{I}$, we write $A=[a^l,a^u]$ where $a^l$ and $a^u$ are lower and upper bounds of $A$ respectively. For $A,B \in \mathbb{I}$ and $k \in \mathbb{R}$, we have
$$A+B=[a^l+b^l,a^u+b^u]$$
$$kA=\begin{cases}
	[ka^l,ka^u], & k \geq 0;\\
	[ka^u,ka^l], & k<0.
\end{cases}$$
From the above two expressions, one has
$$-A=[-a^u,-a^l] ~~ \text{and}~~ A-B=[a^l-b^u,a^u-b^l]$$
The Hausdorff distance between $A$ and $B$ is 
\begin{equation} \label{metric,p2}
	d_H(A,B) =\text{max}\{|a^l-b^l|,|a^u-b^u|\}.
\end{equation}
For more details, we refer to \cite{alefeld,moore2,wu1}.

We can also represent an interval $A\in \mathbb{I}$ in terms of its center and half width (radius) as
\begin{equation} \label{eq1,p2}
	A=\langle a^c,a^w \rangle,
\end{equation}
where $a^c=\frac{a^l+a^u}{2}$ and $a^w=\frac{a^u-a^l}{2}$ are respectively the center and half-width of $A$. Throughout the paper, we will consider the representation (\ref{eq1,p2}) of an interval $A\in \mathbb{I}$. 

The order relation between two intervals in $\mathbb{I}$ is a partial order given by
\begin{equation}\label{eq2,p2}
	A \preceq_{lu} B \text{~~if and only if~ } a^l\leq b^l ~ \text{\&} ~ a^u\leq b^u.
\end{equation}
But, the order relation (\ref{eq2,p2}) in $\mathbb{I}$ is not a total order relation meaning that any two intervals in $\mathbb{I}$ are not comparable. For example, choose $A=[1,4]$ and $B=[2,3]$ then $a^l<b^l$ but $a^u> b^u$. which implies $A$ and $B$ are not comparable with respect to order relation \ref{eq2,p2}. Hence, it is not a total order relation.\\
In view of the above discussion, Bhunia et. al. \cite{bhunia} proposed the following order relations for a minimization problem:

	For any $A,B \in \mathbb{I}$ with $A=\langle a^c,a^w\rangle$ and $B=\langle b^c,b^w\rangle$, we say $A$ is superior (or more preferable) to $B$ in a minimization problem if and only if center of $A$ is strictly less than center of B and half-width (radius), which measures uncertainty (or inexactness), of $A$ is less than or equal to $B$ i.e.,
	\begin{equation} \label{orderrelation,p2}
		A \leq^{\text{min}} B \text{~~if and only if~~} \begin{cases}
			a^c<b^c, & a^c \neq b^c;\\
			a^w\leq b^w, & a^c=b^c.
		\end{cases}
	\end{equation}
	$$~~A <^{\text{min}} B \text{~~if and only if~~} A \leq^{\text{min}} B \text{~and~} A\neq B.$$
	
One can easily verify that the relation given by expression (\ref{orderrelation,p2}) is a total order relation. Through out the paper, we will consider the order relation in $\mathbb{I}$ given by expression (\ref{orderrelation,p2}).\vspace{0.15cm}


\begin{lemma} \label{lemma4,p2}
	For any $A,B\in \mathbb{I}$ with $A=\langle a^c,a^w\rangle$ and $B=\langle b^c,b^w\rangle$, and any $\alpha, \beta \in \mathbb{R}$, we have
	$$\alpha A + \beta B =\langle\alpha a^c + \beta b^c, |\alpha|a^w +|\beta|b^w\rangle.$$
\end{lemma}
\begin{proof}
	Considering different possibilities of $\alpha$ and $\beta$, a total of four cases arise:\\
	{Case (1)} $\alpha,~ \beta \geq 0$, then
	$$\alpha A + \beta B = [\alpha a^l + \beta b^l,~ \alpha a^u + \beta b^u],$$
	which gives the center and half width of $\alpha A + \beta B$ as follows:
	$$(\alpha A + \beta B)^c = \frac{(\alpha a^l + \beta b^l) + (\alpha a^u + \beta b^u)}{2}=\alpha a^c + \beta b^c,$$
	$$\text{and}~~~ (\alpha A + \beta B)^w = \frac{(\alpha a^u + \beta b^u) - (\alpha a^l + \beta b^l)}{2}=|\alpha| a^w + |\beta| b^w.$$
	$$i.e., ~~~ \alpha A + \beta B = \langle\alpha a^c + \beta b^c,~|\alpha| a^w + |\beta| b^w\rangle.$$
	{Case (2)} $\alpha \geq 0$ and $\beta \leq 0$, then
	$$\alpha A + \beta B = [\alpha a^l + \beta b^u,~ \alpha a^u + \beta b^l],$$
	which gives the center and half width of $\alpha A + \beta B$ as follows:
	$$(\alpha A + \beta B)^c = \frac{(\alpha a^l + \beta b^u) + (\alpha a^u + \beta b^l)}{2}=\alpha a^c + \beta b^c,$$
	$$\text{and}~~~ (\alpha A + \beta B)^w = \frac{(\alpha a^u + \beta b^l) - (\alpha a^l + \beta b^u)}{2}= \alpha a^w - \beta b^w = |\alpha| a^w + |\beta| b^w.$$
	$$i.e., ~~~ \alpha A + \beta B = \langle\alpha a^c + \beta b^c,~|\alpha| a^w + |\beta| b^w\rangle.$$
	The following cases are now similar to above two cases:\\
	{Case (3)} $\alpha \leq 0$ and $\beta \geq 0$.\\
	{(Case 4)} $\alpha, ~ \beta \leq 0$.
\end{proof}
\vspace{0.15cm}

 A function $f:E \rightarrow \mathbb{I}$ defined on a subset $E\subseteq (M,g)$ is called an interval-valued function (IVF) denoted by $f(x)=\langle f^c(x),f^w(x)\rangle$, where $f^c(x)$ (center function) and $f^w(x)$ (half-width or radius function) are both real-valued functions defined on $E$, and satisfies $f^w(x) \geq0 ~ \forall x \in E.$ \vspace{0.2cm}

\begin{definition}
	\rm An IVF $f:E \rightarrow \mathbb{I}$, defined on a nonempty subset $E\subseteq \mathbb{R}$, is said to be non-decreasing (increasing) if for any $x,y \in E$, we have 
	$$x<y ~~ \Rightarrow ~ f(x)\leq^{\text{min}} f(y) ~ (f(x)<^{\text{min}} f(y)).$$
\end{definition}

One can similarly define non-increasing and decreasing function.\vspace{0.15cm}

\noindent Since the order relation (\ref{orderrelation,p2}) is a total order relation, we can define infimum and supremum of a subset $\mathbb{A}\subseteq \mathbb{I}$.

\begin{definition}
	\rm We define the infimum and supremum of a subset $\mathbb{A} \subseteq \mathbb{I}$ as follows:
	\begin{itemize}
		\item [\textbullet] A set $\mathbb{A}\subseteq \mathbb{I}$ is bounded below (bounded above) if there exists an interval $B\in \mathbb{I}$ such that
	$$B \leq^{\text{min}} A ~ (A \leq^{\text{min}} B),~~ \text{for all}~~ A \in \mathbb{A}.$$
	The interval $B$ is called lower bound (upper bound) of $\mathbb{A}$. In case no such $B$ exists, the set $\mathbb{A}$ is said to be unbounded below (unbounded above).
	\item [\textbullet] A set $\mathbb{A}$ is bounded if it is bounded above as well as bounded below.
	\item [\textbullet] A member $B$ of a set $\mathbb{A} \subseteq \mathbb{I}$ is least (greatest) if 
	\begin{itemize}
		\item [i)] $B\in \mathbb{A}$
		\item [ii)] $B \leq^{\text{min}}A ~(A \leq^{\text{min}} B) ~~ \forall ~ A \in \mathbb{A}.$
	\end{itemize}
\item [\textbullet] An infimum (supremum) $M \in \mathbb{I}$ of a set $\mathbb{A}\subseteq \mathbb{I}$ is the greatest lower bound (least upper bound) if exists, and we write 
$$M=\text{inf}~ \mathbb{A}~(M=\text{sup} ~\mathbb{A}).$$
	\end{itemize}
\end{definition}

\begin{remark}
\rm	The infimum or supremum of a set $\mathbb{A} \subseteq \mathbb{I}$ may or may not exist. In case it exists, it may or may not belong to the set.
\end{remark}

\begin{example}
  \rm	Consider $\mathbb{A}=\{\langle x,1\rangle \in \mathbb{I}: x \in \mathbb{N}\}$ and $\mathbb{B}=\{\langle x,x+1\rangle \in \mathbb{I}: x \in [0,1)\}$. Then
  \begin{itemize}
  	\item [\textbullet] $\mathbb{A}$ is bounded below but not bounded above, and inf $\mathbb{A}$ = $\langle 1,1\rangle $ $\in \mathbb{A}$;
  	\item [\textbullet] $\mathbb{B}$ is bounded, and inf $\mathbb{A}$ = $\langle 0,1\rangle$ $\in \mathbb{A}$ but sup $\mathbb{B} = \langle 1,2\rangle$ doesn't belong to $\mathbb{B}$.
  \end{itemize} 
\end{example}

\begin{remark}\label{remark2,p2}
	\rm The order completeness property (or least upper bound property) in the totally ordered set $(\mathbb{I},\leq^{\text{min}})$ doesn't hold, meaning that not every nonempty subset of $\mathbb{I}$ which is bounded above has a supremum. For example, consider the following set:
	$$\mathbb{A} = \left\{\langle 0, ~x \rangle \in \mathbb{I} ~:~x\in \mathbb{R}_{+},~ x\geq 0 \right\}.$$
	Then $\mathbb{A}$ is bounded above but has no supremum. Hence, the order relation $\leq^{\text{min}}$ is not complete.
\end{remark}

We now present the definition of limit of an IVF at a point. 
\begin{definition} \label{deflimit,p2}
\rm	Suppose that an IVF $f(x)=\langle f^c(x),f^w(x)\rangle$ be defined on a neighborhood of $x_0 \in (M,g)$ except possibly at $x_0$. We say that $f(x)$ has limit $B\in \mathbb{I}$ at $x_0$ if for every $\epsilon >0$  there exists $\delta>0$ such that 
	$$d(x,x_0) < \delta~~\Rightarrow d_H(f(x),B)<\epsilon,$$ 
	and we write 
	$$\lim\limits_{x\rightarrow x_0} f(x)=B.$$
\end{definition}

The following lemma assures that the limit of an IVF function at a point of $(M,g)$ is equivalent to the limit of its center and half-width (radius) function at that point.

\begin{lemma} \label{lemmalimit,p2}
	Suppose that an IVF $f(x)=\langle f^c(x),f^w(x)\rangle$ be defined on a neighborhood of $x_0 \in (M,g)$ except possibly at $x_0$. Let $B\in \mathbb{I}$ with $B=\langle b^c,b^w\rangle$ be an interval in $\mathbb{I}$, then we have
	$$\lim\limits_{x\rightarrow x_0} f(x)=B ~ \Longleftrightarrow ~ \begin{cases}
		\lim\limits_{x\rightarrow x_0}f^c(x)=b^c;\\
		\lim\limits_{x\rightarrow x_0}f^w(x)=b^w.
	\end{cases}$$
\end{lemma}
\begin{proof}
       The proof is obvious.
\end{proof}

Next, we give the notions of convexity of a real-valued function and an IVF.

\begin{definition} \rm
	\cite{udriste} A subset $E \in (M,g)$ is {totally convex} if $E$ contains every geodesic $\gamma_{\1}$ of $M$ whose end points $x$ and $y$ are in $E$.
\end{definition} 

\begin{definition} \rm 
	\cite{udriste} Let $E$ be totally convex set in $(M,g)$ and $f:E\rightarrow\mathbb{R}$ be a real-valued function. Then:
	\begin{itemize}
		\item [1)] $f$ is convex on $E$, if
		$$f(\gamma_{\1}(s)) \leq (1-s)f(x) +sf(y), ~~~ \forall~ x,y \in E,~~ \gamma_{\1}\in \Gamma,~~\forall~s\in[0,1],$$
		where $\Gamma$ is the collection of all geodesics joining $x$ and $y$.
		\item [2)] $f$ is strictly convex on $E$, if
		$$f(\gamma_{\1}(s)) < (1-s)f(x) +sf(y), ~~~ \forall~ x,y \in E, ~~x\neq y,~~ \gamma_{\1}\in \Gamma,~~\forall~s\in(0,1).$$
	\end{itemize}
\end{definition}

\begin{definition} \rm 
	An IVF $f:E\rightarrow\mathbb{I}$ with $f(x)=\langle f^c(x),~f^w(x)\rangle$, defined on a totally convex set $E\subseteq (M,g)$, is cw-convex (strictly cw-convex) on $E$ if $f^c$ and $f^w$ are convex (strictly convex) on $E$.
\end{definition}

\begin{definition}
	\rm \cite{udriste} A set $E\subseteq (M,g)$ is star-shaped at $x_0 \in E$ if $\gamma_{\2}(s)\in E$ whenever $x\in E$ and $s \in (0,1)$, where $\gamma_{\2}$ is any geodesic in $E$ joining $x_0$ with $x$.
\end{definition}

\begin{definition} \rm 
	\cite{udriste} Let $E \subseteq (M,g)$ be star-shaped at $x_0 \in E$ and $f:E\rightarrow\mathbb{R}$ be a real-valued function. Then:
	\begin{itemize}
		\item [1)] $f$ is convex at $x_0$, if
		$$f(\gamma_{\2}(s)) \leq (1-s)f(x_0) +sf(x), ~~~ \forall~ x \in E,~~ \gamma_{\2}\in \Gamma_0,~~\forall~s\in(0,1),$$
		where $\Gamma_0$ is the collection of all geodesics emanating from $x_0$ and terminating at $x$.
		\item [2)] $f$ is strictly convex at $x_0$, if
		$$f(\gamma_{\2}(s)) < (1-s)f(x_0) +sf(x), ~~~ \forall~ x \in E, ~~x\neq x_0,~~ \gamma_{\2}\in \Gamma_0,~~\forall~s\in(0,1).$$
	\end{itemize}
\end{definition}

\begin{definition} \label{definition**,p2}
\rm	Let $E\subseteq (M,g)$ be star-shaped at $x_0 \in E$. We say an IVF $f:E\rightarrow \mathbb{I}$ with $f(x)=\langle f^c(x),~f^w(x)\rangle$ is cw-convex at $x_0$, if $f^c$ and $f^w$ are convex at $x_0$.
\end{definition}

\section{Directional differentiability of a real-valued function}
   ~~~In this section, we present the concept of directional differentiability for a real-valued function and some of the basic results related to both convexity and directional differentiability of a function. Furthermore, we give the counter examples to those arguments that remain invalid when the convexity is assumed at a single point of the domain. We extend these results for an IVF in the next section.

 If $E \subseteq M$ is totally convex, then $T_x(E)$, the set of tangent vectors to $E$ at any $x\in E$, is a convex cone. Furthermore, if $E$ is star-shaped at some $x_0\in E$, then $T_{x_0}(E)$, the set of tangent vectors to $E$ at $x_0$, is star-shaped at $0$.\vspace{0.15cm}

\begin{definition} \label{definition5.1,p2} \rm 
	Let $E$ be a subset of $(M,g)$ and $x\in E$. Let $X_x \in T_x(E)$ and $\gamma(s);~s\in I, ~0\in I~ \text{and}~ \gamma(I)\subseteq E$, be a geodesic for which $\gamma(0)=x, ~ \&~ \dot{\gamma}(0)=X_x$. We say a real-valued function $f:E \rightarrow \mathbb{R}$ is directionally differentiable at $x$ in the direction $X_x$, if the limit
	$$Df(x;~X_x)=\lim\limits_{s\rightarrow0^+} \frac{f(\gamma(s))-f(x)}{s}$$
	exists, where $Df(x;X_x)$ is said to be directional derivative of $f$ at $x$ in the direction $X_x$.  Moreover, we say $f$ is directionally differentiable at $x$, if $Df(x;X_x)$ exists at $x$ in every direction $X_x \in T_x(E)$. Furthermore, if $Df(x;X_x)$ exists at each $x\in E$ and in every direction $X_x \in T_x(E)$, we say $f$ is directionally differentiable on $E$.
\end{definition}

\begin{example}
	\rm Let $M= \{(x,y,z): x^2+y^2=1, z\in \mathbb{R}\}\subseteq \mathbb{R}^3$ be the cylinder of radius 1 along Z-axis. Let $x_0=(1,0,0)$, and $E=E_1 \cup E_2$ where,
	 \begin{align*}
	 	E_1&=\{(x,y,z)\in M: x=1, y=0, z\in \mathbb{R}\}\\
	 	\text{and} ~~~E_2&=\{(x,y,z)\in M: x^2+y^2=1, z=0\}.
	 \end{align*}
	Then $T_{x_0}(E)$ is the union of two orthogonal lines of $\mathbb{R}^2$ passing through origin. For the sake of simplicity, we can choose 
	$$T_{x_0}(E)=\{(0,v_1,v_2)\in \mathbb{R}^3: \text{~either}~ v_1=0 ~ \text{or} ~v_2=0\},$$
	 which is the union of $Y$-axis and $Z$-axis. Clearly, $T_{x_0}(E)$ is star-shaped at $(0,0)$. 
	 
	 Define a function $f: E \rightarrow\mathbb{R}$ as follows:
	 $$f(x,y,z) = \begin{cases}
	 	z^2+z+1, & (x,y,z)\in E_1\\
	 	e^{1-x^2}, & (x,y,z)\in E_2.
	 \end{cases}$$
 The geodesics emanating from $x_0=(1,0,0)$ in any direction $X_{x_0}=(0,v_1,v_2) \in T_{x_0}(E)$ are given by
 	\begin{align*}
 		\gamma_{\4}(s) &= (1,0,sv_2) \in E_1, ~~ s\in I, ~v_1=0,
 	\intertext{where $I$ is interval in $\mathbb{R}$ containing $0$ and $\gamma_{\4}(I) \subseteq E_1$, and} 
 	\gamma_{\5}(s) &= (\cos(v_1s), \sin (v_1s),0) \in E_2, ~~ s\in \left[0,\frac{2\pi}{|v_1|}\right], v_2=0, v_1\neq 0.
 	\end{align*}
 The directional derivative of $f$ at $x_0=(1,0,0)$ in any direction $T_{x_0}(E)$ is given by
 $$Df(x_0;~X_{x_0})) = \begin{cases}
 	v_2, &v_1 = 0;\\
 	0, & v_1\neq 0,~v_2=0.
 \end{cases}$$
\end{example}

In view of Definition \ref{definition5.1,p2}, we have the following lemma.

\begin{lemma} \label{lemmaf+g,p2}
	Suppose that $f,g:E\rightarrow \mathbb{R}$, $E\subseteq (M,g)$, be directionally differentiable at $x\in E$ in the direction $X_x\in T_x(E)$. Then $\alpha_1 f\pm \alpha_2 g$, $\alpha_1$, $\alpha_2$ $\in$ $\mathbb{R}$, is directionally differentiable at $x\in E$ in the direction $X_x\in T_x(E)$.
\end{lemma}
\begin{proof}
The proof is obvious.
\end{proof}

In view of Definition \ref{definition5.1,p2} and Definition 4.1 (Chapter 3 in \cite{udriste}), the following theorem follows directly from Theorem 4.2 (Chapter 3 in \cite{udriste}).

\begin{theorem} \label{theorem5.1,p2}
	If $f: E \rightarrow \mathbb{R}$ is convex on $E$ and $\gamma(s); ~s\in I, ~0\in I~ \&~ \gamma(I)\subseteq E$, is the geodesic for which $\gamma(0)=x~ \text{and} ~ \dot{\gamma}(0)=X_x$, then
	\begin{itemize}
		\item [(i)] for a fixed $X_x$, the function $Q:I\cap (0,\infty) \rightarrow \mathbb{R}$ defined by
		$$Q(s)=\frac{f(\gamma(s))-f(x)}{s},$$
		is non-decreasing;
		\item[(ii)] $Df(x;~X_{x})$ exists and $Df(x;~X_{x})= \underset{s}{\text{inf}} ~Q(s)$;
		\item [(iii)] $Df(x;~X_{x})$ is convex and positively homogeneous
		$$i.e., ~~ Df(x;~\lambda X_{x})= \lambda Df(x;~X_{x}), ~~ for ~~ \lambda\geq 0;$$
		\item[(iv)] $Df(x;~0)=0$ and -$Df(x;-X_{x})\leq$ $Df(x;~X_{x}).$
	\end{itemize}
\end{theorem}

	In Theorem \ref{theorem5.1,p2}, part (ii) may not hold true if we assume the convexity of $f$ at a single point. For counter example one can refer to Example \ref{example6.1,p2} in Appendix A. Hence, we have the following Theorem:
	
	\begin{theorem} \label{theorem16*,p2}
		Let $E\subseteq (M,g)$ be star-shaped at $x_0 \in E$ and $\gamma(s);~s\in I, ~0\in I~ \&~ \gamma(I)\subseteq E$, be the geodesic for which $\gamma(0)=x_0, ~ \&~ \dot{\gamma}(0)=X_{x_0}$. If $f: E \rightarrow \mathbb{R}$ is convex at $x_0$, then
		\begin{itemize}
			\item [(i)] for a fixed $X_{x_0} \in T_{x_0}(E)$, the function $Q:I\cap(0,\infty) \rightarrow \mathbb{R}$, defined by
			$$Q(s)=\frac{f(\gamma(s))-f(x_0)}{s},$$
			is non-decreasing;
			\item [(ii)] if $Df(x_0;~X_{x_0})$ exists, then 
			\begin{itemize}
				\item [(a)] $Df(x_0;~X_{x_0})= \underset{s}{\text{inf}} ~Q(s)$;
				\item [(b)] the real-valued function $g(X_{x_0}) = Df(x_0;~X_{x_0})$, defined on $T_{x_0}(E)$ is positively homogeneous.
			\end{itemize}
		\end{itemize}
	\end{theorem}
	\begin{proof}
		Parts (i) and (ii)(a) are similar to the respective parts of Theorem \ref{theorem5.1,p2}.\\
		(ii)(b) For positive homogeneity of $g$, we let $s>0$ sufficiently small, such that for any $\lambda\geq0$, we have
		\begin{align*}
			g(\lambda X_{x_0}) &= Df(x_0;~\lambda X_{x_0})\\
			&= \lim\limits_{s\rightarrow0^+}\frac{\gamma(\lambda s)-f(x_0)}{s}\\
			&= \lim\limits_{s\rightarrow0^+}\frac{f(\gamma(u))-f(x_0)}{s},~~ \lambda s=u,\\
			&= \lambda Df(x_0;~X_{x_0}).
		\end{align*} 
	\end{proof}

\begin{remark} \rm 
	In Theorem \ref{theorem16*,p2}, if for any $X_{x_0} \in T_{x_0}(E)$, we have that $-X_{x_0} \in T_{x_0}(E)$, and also $Df(x_0;~ X_{x_0})$ and $Df(x_0; -X_{x_0})$ both exist, then
	$$-Df(x_0; -X_{x_0})\leq Df(x_0;~ X_{x_0})$$
	may not hold true in general. This is evident from the following example.
\end{remark}

\begin{example} \rm 
	Define $f: \mathbb{R} \rightarrow \mathbb{R}$ as
	$$f(x) = \begin{cases}
		0, &x\leq 0;\\
		-x, &x>0.
	\end{cases}$$
Here $f$ is convex at $x=0$. Let $x_0=1,$ then $Df(0;~1)=-1$ and $Df(0;-1)=0$. Clearly, $-Df(0;-1)\geq Df(0;~1).$
\end{example}

\begin{remark} \rm 
	In Theorem \ref{theorem16*,p2}, the function $g$ may fail to be convex on $T_{x_0}(E)$, even if we further restrict $E$ to be totally convex set. More precisely, Let $E$ be totally convex set and $f:E\rightarrow \mathbb{R}$ be directionally differentiable at $x_0$. If $f$ is convex at $x_0$, then $g(X_{x_0})=Df(x_0; X_{x_0})$ may fail to be convex on $T_{x_0}(E)$. The following example supports our argument.
\end{remark}
	
\begin{example}
	\rm Define $f: \mathbb{R} \rightarrow \mathbb{R}$ as
	$$f(x) = \begin{cases}
		1, &x\leq 0;\\
		-x+1, &x>0.
	\end{cases}$$
Here, both $E$ and $T_{x_0}(E)=\mathbb{R}$, $x_0=0$, are totally convex sets. $f$ is convex only at $x_0=0$. One can calculate that, for any $y\in T_{x_0}(E)=\mathbb{R}$,
$$g(y)= \begin{cases}
	0, &y\leq 0;\\
	-y, &y>0.
\end{cases}$$
Here, $g$ is convex at $y=0$ but fails to be convex on $T_{x_0}(E)=\mathbb{R}$.
\end{example}

Next, the following Theorem gives a necessary and sufficient condition for a real valued directionally differentiable function to be convex.

\begin{theorem}\cite{udriste} \label{theorem15,p2}
Suppose that the function $f: E \rightarrow \mathbb{R}$ is defined on a totally convex set $E\subseteq(M,g)$. Then
\begin{enumerate} 
	\item [(i)] $f$ is convex on $E$ if and only if 
\begin{equation} \label{eq9*,p2}
		f(y)-f(x)\geq Df(x;~X_x); ~~\forall~x,y \in E,~~ \forall~\gamma_{\1}\in \Gamma,
\end{equation} 
	where $\Gamma$ is the collection of geodesics joining $x$ and $y$ with $\gamma_{\1}(0)=x$ and $\dot{\gamma}_{\1}(0)=X_x$.
	\item [(ii)] $f$ is strictly convex on $E$ if and only if 
	$$f(y)-f(x)> Df(x;~X_x); ~~\forall~ x,y \in E,~x\neq y,~~ \forall~\gamma_{\1}\in \Gamma.$$
\end{enumerate}
\end{theorem}

The sufficient part of Theorem \ref{theorem15,p2} is not true in general when the convexity of the function is assumed at a single point. Hence,  we have the following Theorem

\begin{theorem}\cite{udriste} \label{theorem30,p2}
	Let $E \subseteq (M,g)$ be star-shaped at $x_0 \in E$ and the function $f:E\rightarrow \mathbb{R}$ be directionally differentiable at $x_0$, then,
	\begin{enumerate} 
		\item [(i)] if $f$ is convex at $x_0$, then 
		\begin{equation}
			f(x)-f(x_0)\geq Df(x_0;~X_{x_0}); ~~\forall~x \in E,~~ \forall~\gamma_{\2}\in \Gamma_0,
		\end{equation} \label{eq9,p2}
		where $\Gamma_0$ is the set of geodesics joining $x_0$ and $x$ such that $\gamma_{\2}(0)=x$ and $\dot{\gamma}_{\2}(0)=X_{x_0}$,
		\item [(ii)] if $f$ is strictly convex at $x_0$ then 
		$$f(x)-f(x_0)> Df(x_0;~X_{x_0}); ~~\forall~y \in E,~~ y\neq x_0~~ \forall~\gamma_{\2}\in \Gamma_0.$$
	\end{enumerate}
\end{theorem}

In general, the converse of Theorem \ref{theorem30,p2} cannot be assumed to hold true. This can be illustrated by the following example.

\begin{example}\label{exg3.4}
	\rm Let $\mathit{M}=\{e^{i\theta} : \theta \in [0,2\pi]\}$ be the unit circle in complex plane which is a compact $1-$dimensional Riemannian manifold$^*$.
	\let\thefootnote\relax\footnotetext{\footnotesize
		$^*$ In this case, we assume that the manifold \( \mathit{M} = \{e^{i\theta} : \theta \in [0,2\pi]\} \) is endowed with the identification \( 0 \sim 2\pi \), ensuring that the endpoints of the interval correspond to the same point in \( \mathit{M} \). As a result, \( \mathit{M} \) is diffeomorphic to the unit circle \( S^1 =\{(x,y)\in \mathbb{R}^2: x^2+y^2=1\} \), forming a compact 1-dimensional Riemannian manifold. The coordinate \(\theta\) provides a global parameterization within \( [0,2\pi] \), with the constraint that \(\theta = 0\) and \(\theta = 2\pi\) are identified to maintain the circular topology.
	}	 
	Fix $x_0 = e^{i\frac{\pi}{2}}=i \in M$. The geodesic segment joining $x_0=i$ with any $x=e^{i\theta}\in M$ is given by
$$\gamma(s) = e^{i(\frac{\pi}{2}+s(\theta -\frac{\pi}{2}))}, ~~~ s\in [0,1],$$
where $\gamma(0)=i$ and $\dot{\gamma}(0)=\frac{\pi}{2}-\theta=X_{x_0}\in \mathbb{R}$.

Define $f:M\rightarrow \mathbb{R}$, by
$$f(x)=f(e^{i\theta}) = \ln((\theta - \frac{\pi}{2})^2 +e).$$
The directional derivative of $f$ at $x_0=i$ in any direction $X_{x_0}=\frac{\pi}{2}-\theta$  is 
$$Df(x_0;~X_{x_0})=0.$$
For any $x=e^{i\theta}\in M$, we have 
$$f(x)-f(x_0)=\ln((\theta -\frac{\pi}{2})^2+e)-1$$
One can easily see that
$$f(x)-f(x_0)\geq Df(x_0;~X_{x_0}).$$
But $f$ fails to be convex at $x_0=i$, as one can check the following relation
$$f(\gamma(\frac{1}{2}))\nleq\frac{1}{2}f(i)+\frac{1}{2}f(-i)$$
Hence, in general, converse of Theorem \ref{theorem30,p2} is not true.
\end{example}

\section{gH-directional differentiability of an IVF}

In this section, we discuss the directional differentiability for an IVF. We show that the necessary part of the Lemma 3.2 in \cite{chen} is not true in general. 
\vspace{0.15cm}

The generalized Hukuhara difference (gH-difference) of two intervals $A$, $B$ $\in \mathbb{I}$, denoted by $A \ominus_{gH} B$, was introduced by Stefanini et. al. \cite{stefanini}. This can be expressed as:
$$A \ominus_{gH} B = [\text{min}\{a^l-b^l, a^u-b^u\}, \text{max}\{a^l-b^l, a^u-b^u\}].$$
 It's worth noting that for any two intervals $A, B\in \mathbb{I}$, $A \ominus_{gH} B$ always exists and is unique. Also, $A\ominus_{gH} A = [0,0].$

The following lemma expresses the gH-difference of two intervals in $\mathbb{I}$ in terms of their center and half-width.

\begin{lemma}
	\label{lemma13,p2}
	For any $A,B \in \mathbb{I}$ with $A=[a^l,a^u]=\langle a^c,a^w\rangle$ $\&$ $B=[b^l,b^u]=\langle b^c,b^w\rangle$, we have
	$$A\ominus_{gH} B~ = ~\langle a^c-b^c,~|a^w-b^w|\rangle$$
\end{lemma}

\begin{proof} The gH-difference of $A$ and $B$ in terms of left and right bounds is given by
	\begin{align*}
		A \ominus_{gH} B &= [\text{min}\{a^l-b^l, a^u-b^u\}, \text{max}\{a^l-b^l, a^u-b^u\}]= [p,~q]\\ 
		\text{where}~~ p &=\frac{a^l-b^l+a^u-b^u}{2}-\frac{|a^u-b^u-a^l+b^l|}{2}\\
		\text{and}~~ q &= \frac{a^l-b^l+a^u-b^u}{2}+\frac{|a^u-b^u-a^l+b^l|}{2}\\
	\Rightarrow~~ 	A \ominus_{gH} B &= [a^c-b^c -|a^w-b^w|,~a^c-b^c +|a^w-b^w|].\\
	   &= \langle a^c-b^c,~|a^w-b^w|\rangle.
	\end{align*}
\end{proof}

\begin{definition}\label{Definition4.2,p2}\rm 
	Let $E$ be a subset of $(M,g)$ and $x\in E$. Let $X_x \in T_x(E)$ and $\gamma(s);~s\in I, ~0\in I~ \text{and}~ \gamma(I)\subseteq E$, be a geodesic with $\gamma(0)=x, ~ \text{and}~ \dot{\gamma}(0)=X_x$. We say an IVF $f:E \rightarrow \mathbb{I}$ with $f(x)= \langle f^c(x),~f^w(x)\rangle$ is gH-directionally differentiable at $x$ in the direction $X_x$, if the following limit
	$$Df(x;~X_x)=\lim\limits_{s\rightarrow0^+} \frac{f(\gamma(s))\ominus_{gH}f(x)}{s}$$
	exists, where $Df(x;X_x)$ is said to be gH-directional derivative of $f$ at $x$ in the direction $X_x$. Moreover, we say $f$ is gH-directionally differentiable at $x$, if $Df(x;X_x)$ exists at $x$ in every direction $X_x \in T_x(E)$. Furthermore, if $Df(x;X_x)$ exists at each $x\in E$ and in every direction $X_x \in T_x(E)$, we say $f$ is gH-directionally differentiable on $E$.
\end{definition}

The following lemma provides a condition that is enough to ensure that an IVF function can be gH-directionally differentiable.

\begin{lemma} \label{lemma18,p2}
	Let $f:E\rightarrow \mathbb{R}$ with $f(x)= \langle f^c(x),~f^w(x)\rangle$ be an IVF defined on $E\subseteq (M,g)$ and $\gamma(s);~s\in I, ~0\in I~ \text{and}~ \gamma(I)\subseteq E$, be a geodesic for which $\gamma(0)=x\in E, ~ \&~ \dot{\gamma}(0)=X_x\in T_{x}(E)$. Suppose that directional derivative of $f^c$ and $f^w$ exists at $x$ in the direction $X_x$, then gH-directional derivative of $f$ at $x$ in the direction $X_x$ also exists and hence
	\begin{equation} \label{eq**,p2}
		Df(x;X_x)~=~\langle Df^c(x;X_x),~|Df^w(x;X_x)|\rangle,
	\end{equation}
where $Df^c(x;X_x)$ and $Df^w(x;X_x)$ refer to the directional derivatives of $f^c$ and $f^w$, respectively, at $x$ in the direction $X_x$.
\end{lemma}

\begin{proof}
      \begin{align*}
      	Df(x;X_x) &= \lim\limits_{s\rightarrow0^+} \frac{f(\gamma(s))\ominus_{gH} f(x)}{s},\\
      	&= \lim\limits_{s\rightarrow0^+}\frac{\left \langle f^c(\gamma(s)),~f^w(\gamma(s))\right \rangle \ominus_{gH} \left \langle f^c(x),~f^w(x)\right \rangle}{s}\\
      	\intertext{which by Lemma \ref{lemma13,p2} and Lemma \ref{lemma4,p2}, yields}
     Df(x;X_x) &= \lim\limits_{s\rightarrow0^+} \left\langle \frac{f^c(\gamma(s))- f^c(x)}{s},~\frac{|f^w(\gamma(s))- f^w(x)|}{s}\right\rangle.\\
     \intertext{Using Lemma \ref{lemmalimit,p2}, we get}
     Df(x;X_x) &=  \left\langle \lim\limits_{s\rightarrow0^+}\frac{f^c(\gamma(s))- f^c(x)}{s},~\lim\limits_{s\rightarrow0^+}\frac{|f^w(\gamma(s))- f^w(x)|}{s}\right\rangle,\\
     \intertext{which by continuity of modulus function and non-negativity of s, yields}
     Df(x;X_x) &=  \left\langle \lim\limits_{s\rightarrow0^+}\frac{f^c(\gamma(s))- f^c(x)}{s},~\left|\lim\limits_{s\rightarrow0^+}\frac{f^w(\gamma(s))- f^w(x)}{s}\right|\right\rangle,\\
     \text{i.e., ~~~~~} Df(x;X_x) &= \langle Df^c(x;X_x),~|Df^w(x;X_x)|\rangle. 
      \end{align*}
\end{proof}

\begin{remark} \label{remark4.1,p2}\rm 
	Suppose that an IVF $f:E\rightarrow \mathbb{I}$ with $f(x)= [f^l(x), f^u(x)] = \langle f^c(x), f^w(x) \rangle $, defined on subset $E\subseteq (M,g)$, is gH-directionally differentiable at $x\in E$ in the direction $v \in T_x(E)$. Then, from Lemma 3.2 in \cite{chen}, both $f^l$ and $f^u$ are directionally differentiable at $x\in E$ in the direction $v \in T_x(E)$ and hence, using Lemma \ref{lemmaf+g,p2}, $f^w = \frac{f^u-f^l}{2}$ is directionally differentiable at $x\in E$ in the direction $v \in T_x(E)$. But the Example \ref{example20,p2}, which follows next, shows otherwise, i.e., the necessary part of Lemma 3.2 in \cite{chen}  and the converse of Lemma \ref{lemma18,p2} remain invalid in general.
\end{remark}

\begin{example} \label{example20,p2}
	\rm Consider the Riemannian manifold \( \mathit{M} \) as defined in Example \ref{exg3.4} and choose $x = e^{i\frac{\pi}{2}}=i$ and $y=e^{i\pi}=-1$. The geodesic segment joining $x=i$ and $y=-1$ is given by
	$$\gamma_{\1}(s) = e^{i\frac{\pi}{2}(1+s)}, ~~~ s\in [0,1],$$	
	and the  tangent vector to the geodesic is  $X_x=\dot{\gamma}_{\1}(0)=-\frac{\pi}{2}\in \mathbb{R}.$
	
	Define an IVF $f:M \rightarrow \mathbb{I}$ as following
	\begin{align*}
		f(x) &= ~\langle 0,~f^w(x)\rangle,\\
		\text{where}~~~ f^w(x)=f^w(e^{i\theta})&= \begin{cases}
			1, & x\in M\setminus \gamma_{\1}(s), ~~ s\in [0,1];\\
			\theta, & x\in \gamma_{\1}(s), ~~~~~~~~ s \in [0,1]\cap \mathbb{Q};\\
			\pi - \theta, & x \in \gamma_{\1}(s),  ~~~~~~~~ s \in [0,1]\cap \mathbb{R}\setminus\mathbb{Q}.
		\end{cases}
	\end{align*}
The limit of slope function of secants to $f^w$ at $x=i$ in the direction $X_x=-\frac{\pi}{2}$ is
$$\lim\limits_{s\rightarrow0^+}\frac{f^w(\gamma_{\1}(s)-f^w(x))}{s}= \begin{cases}
	~~\frac{\pi}{2}, & s \in \mathbb{Q};\\
	-\frac{\pi}{2}, & s \in [0,1]\cap \mathbb{R}\setminus\mathbb{Q},
\end{cases}$$
which is not unique and hence $f^w$ is not directionally differentiable at $x=i$ in the direction $X_x=-\frac{\pi}{2}.$

However, $f$ is gH-directionally differentiable at $x=i$ in the direction $X_x=-\frac{\pi}{2}$, which is evident from the following:
	$$\frac{f(\gamma_{\1}(s))\ominus_{gH}f(x)}{s} = 	\frac{f(e^{i\frac{\pi}{2}(1+s)})\ominus_{gH}f(e^{i\frac{\pi}{2}})}{s},$$
	which by Lemma \ref{lemma13,p2}, yields
	\begin{align*}
		\frac{f(\gamma_{\1}(s))\ominus_{gH}f(x)}{s} &= \frac{1}{s}\left\langle f^c(e^{i\frac{\pi}{2}(1+s)})-f^c(e^{i\frac{\pi}{2}}),~ f^w(e^{i\frac{\pi}{2}(1+s)})-f^w(e^{i\frac{\pi}{2}})\right\rangle,\\
		&= \begin{cases}
			\frac{1}{s} \left\langle 0,~|\frac{\pi}{2}(1+s)-\frac{\pi}{2}\right\rangle, & s \in \mathbb{Q};\\
			\frac{1}{s} \left\langle 0,~|\pi-\frac{\pi}{2}(1+s)-\frac{\pi}{2}\right\rangle, & s \in [0,1]\cap \mathbb{R}\setminus\mathbb{Q},
		\end{cases}\\
	&= \begin{cases}
		\left \langle 0,~\frac{\pi}{2}\right \rangle, & s \in \mathbb{Q};\\
		\left \langle 0,~\frac{\pi}{2}\right \rangle, & s \in \mathbb{Q},
	\end{cases}\\
&= \left\langle 0,\frac{\pi}{2}\right\rangle.
	\end{align*}
Hence, the gH-directional derivative of $f$ at $x=i$ in the direction $X_x=-\frac{\pi}{2}$ is 
$$Df(x; X_x) = \lim\limits_{s\rightarrow0^+} \frac{f(\gamma_{\1}(s))\ominus_{gH}f(x)}{s} = \left\langle 0, \frac{\pi}{2}\right\rangle,$$
This shows that the necessary part of Lemma 3.2 in \cite{chen} and the converse of Lemma \ref{lemma18,p2} do not hold true in general.
\end{example}

The following lemma guarantees the existence of directional derivative of center function $f^c$ of a gH-directionally differentiable IVF $f(x)=\langle f^c(x),f^w(x)\rangle$.

\begin{lemma} \label{lemma21,p2}
	Suppose that $f:E\rightarrow \mathbb{R}$ with $f(x)= \langle f^c(x),~f^w(x)\rangle$ is an IVF defined on $E\subseteq (M,g)$ and $\gamma(s);~s\in I, ~0\in I~ \text{and}~ \gamma(I)\subseteq E$, is a geodesic for which $\gamma(0)=x\in E, ~ \&~ \dot{\gamma}(0)=X_x\in T_{x}(E)$. Suppose that gH-directional derivative of $f$ exists at $x$ in the direction $X_x$, then the directional derivative of $f^c$ also exists at $x$ in the direction $X_x$.
\end{lemma}
\begin{proof}
	\begin{align*}
		Df(x;X_x) &= \lim\limits_{s\rightarrow0^+} \frac{f(\gamma(s))\ominus_{gH} f(x)}{s},\\
		&= \lim\limits_{s\rightarrow0^+}\frac{\left \langle f^c(\gamma(s)),~f^w(\gamma(s))\right \rangle \ominus_{gH} \left \langle f^c(x),~f^w(x)\right \rangle}{s}\\
		\intertext{which by Lemma \ref{lemma13,p2} and Lemma \ref{lemma4,p2}, yields}
		Df(x;X_x) &= \lim\limits_{s\rightarrow0^+} \left\langle \frac{f^c(\gamma(s))- f^c(x)}{s},~\frac{|f^w(\gamma(s))- f^w(x)|}{s}\right\rangle.\\
		\intertext{Using Lemma \ref{lemmalimit,p2}, we get}
		Df(x;X_x) &=  \left\langle \lim\limits_{s\rightarrow0^+}\frac{f^c(\gamma(s))- f^c(x)}{s},~\lim\limits_{s\rightarrow0^+}\frac{|f^w(\gamma(s))- f^w(x)|}{s}\right\rangle.
	\end{align*}
Since as $Df(x;X_x)$ exists, the last equation guarantees that 
$$Df^c(x,X_x)=\lim\limits_{s\rightarrow0^+}\frac{f^c(\gamma(s))- f^c(x)}{s},$$
exists. This completes the proof.
\end{proof}

So far now, the inversion of Lemma \ref{lemma18,p2} is not possible in general. The problem arises with the half-width (radius) function (see Example \ref{example20,p2} \& Lemma \ref{lemma21,p2}). However, the inversion is possible for the class of IVF in which the composition of radius function and the geodesic, which emanates from a point in a direction at which the function is gH-directionally differentiable in the same direction, is non-decreasing in [0, $\infty$). The same is addressed in the following lemma.

\begin{lemma} \label{lemma22,p2}
		Let $f:E\rightarrow \mathbb{R}$ with $f(x)= \langle f^c(x),~f^w(x)\rangle$ be an IVF defined on $E\subseteq (M,g)$. Let $x\in E$ and $\gamma(s);~s\in I, ~0\in I~ \text{and}~ \gamma(I)\subseteq E$, be any geodesic for which $\gamma(0)=x, ~ \&~ \dot{\gamma}(0)=X_x\in T_{x}(E)$ such that $(f^w \circ \gamma)(s)$ is non-decreasing for $s \in I\cap[0,\infty)$. Then, gH-directional derivative of $f$ exists at $x$ in the direction $X_x$ if and only if the directional derivative of $f^c$ and $f^w$ exists at $x$ in the direction $X_x$. Hence,
		$$Df(x;X_x)~ = ~\langle Df^c(x;X_x),~Df^w(x;X_x)\rangle.$$
		where $Df^w(x;X_x)\geq 0$.
\end{lemma}
\begin{proof}
	Here we only need to show that directional derivative of $f^w$ exists at $x$ in the direction $X_x$ when gH-directional derivative of $f$ exists at $x$ in the direction $X_x$. The rest of the proof follows from Lemma \ref{lemma18,p2} and Lemma \ref{lemma21,p2}.
	Continuing on the same lines of the proof of Lemma \ref{lemma21,p2}, we have
	$$Df(x,X_x)=\left\langle \lim\limits_{s\rightarrow0^+}\frac{f^c(\gamma(s))- f^c(x)}{s},~\lim\limits_{s\rightarrow0^+}\frac{|f^w(\gamma(s))- f^w(x)|}{s}\right \rangle,$$
	which from the existence of $Df(x,X_x)$ yields the existence of following limit
	$$\lim\limits_{s\rightarrow0^+}\frac{|f^w(\gamma(s))- f^w(x)|}{s}.$$
	Since as $(f^w\circ \gamma)(s)$ is non-decreasing in $I \cap [0, \infty)$, we have
	$$(f^w\circ \gamma)(s) \geq (f^w\circ \gamma)(0),~~~ s>0,$$
	$$\Rightarrow ~~ f^w(\gamma(s)) - f^w(x) \geq 0.$$
	Hence, we have the existence of following
	$$Df^w(x,X_x)=\lim\limits_{s\rightarrow0^+}\frac{f^w(\gamma(s))- f^w(x)}{s} = \lim\limits_{s\rightarrow0^+}\frac{|f^w(\gamma(s))- f^w(x)|}{s}.$$
	This completes the proof.
\end{proof}

The following example is in support of Lemma \ref{lemma22,p2}.

\begin{example}\label{exg4.2} \rm 
	Let $\mathit{M}=\{e^{i\theta} : \theta \in \mathbb{R}\}$ be a non-compact $1-$dimensional Riemannian manifold$^*$
	\let\thefootnote\relax\footnotetext{\footnotesize
		$^*$ In this case, we assume that the manifold \( \mathit{M} = \{e^{i\theta} : \theta \in \mathbb{R}\} \) is not periodic, meaning that distinct values of \(\theta\) correspond to distinct points in \( \mathit{M} \). As a result, \( \mathit{M} \) is diffeomorphic to \(\mathbb{R}\) rather than the unit circle \( S^1 =\{(x,y)\in \mathbb{R}^2: x^2+y^2=1\} \). Unlike the standard compact circle where \(\theta\) is identified modulo \(2\pi\), our construction treats \(\theta\) as a global coordinate extending infinitely in both directions.
	}
	and let $x_0=i$. The tangent space at $x_0=i$ is $T_i(M)=\mathbb{R}$. The geodesic emanating from $x_0=i$ in any direction $v\in \mathbb{R}$ is given by
	$$\gamma(s) = ie^{-isv},~~ s \in I,$$
	where $I$ is an interval in $\mathbb{R}$ containing $0$ such that $\gamma(I)\subseteq M$.
	
	We define $f: M \rightarrow \mathbb{I}$, as 
	$$f(x)=f(e^{i\theta})= \langle 1, ~(\theta - \frac{\pi}{2})^2\rangle, ~~\theta\in [0, 2\pi].$$
	One can calculate that $(f^w \circ \gamma)(s)=s^2v^2$ which is clearly non-decreasing in $I\cap [0,\infty)$. Also, gH-directional derivative of $f$ at $i$ in any direction $v\in \mathbb{R}$ is 
	$$Df(i;v) = \langle 0,0 \rangle=\langle Df^c(i;v), |Df^w(i;v)|\rangle,$$
	where $Df^c(i;v)$ and $Df^w(i;v)$ are respectively the directional derivatives of $f^c$ and $f^w$ at $i$ in any direction $v\in \mathbb{R}$.
\end{example}
\begin{remark}\rm
	On Riemannian manifolds such as circles and spheres, where all geodesics are periodic, there does not exist a function \( f \) such that its composition with any geodesic is strictly non-decreasing. The periodicity of geodesics implies that any such function must be constant, as a strictly non-decreasing function would contradict the periodic nature of the geodesics.
\end{remark}

The following lemma gives an equivalence of gH-directional differentiability of an IVF in terms of its end point functions.
\begin{lemma}
	Under the same assumptions of Lemma \ref{lemma22,p2}, for an IVF with $f(x)=[ f^l(x), f^u(x)]=\langle f^c(x), f^w(x) \rangle$, we have
	$$Df(x;X_x)=[ Df^l(x;X_x), Df^u(x;X_x)],$$
	where $Df^l(x;X_x)=Df^c(x)-Df^w(x)$ and $Df^u(x;X_x)=Df^c(x)+Df^w(x)$.
\end{lemma}
\begin{proof}
The proof can be directly inferred from Lemma \ref{lemma22,p2} and Lemma \ref{lemmaf+g,p2}.
\end{proof}

\begin{lemma} \label{lemma4**,p2}
	Let $E\subseteq (M,g)$ be star-shaped at $x_0 \in E$. Then $f:E \rightarrow \mathbb{R}$ is convex at $x_0$ if and only if for any $x \in E$ the function $f \circ \gamma_{\2}:[0,1]\rightarrow\mathbb{R}$ is convex at 0, where $\gamma_{\2}$ is the geodesic segment joining $x_0$ and $x$.
\end{lemma}
\begin{proof}
	The proof is analogous to Theorem 2.2 of Chapter 3 in \cite{udriste}.
\end{proof}

\begin{lemma} \label{lemmaslope,p2}
	Let $f:[0,1] \rightarrow \mathbb{R}$ be convex at $0$, then, for any $0<s_1<s_2$, we have
	$$\frac{f(s_1)-f(0)}{s_1}\leq \frac{f(s_2)-f(0)}{s_2}$$.
\end{lemma}
\begin{proof}
	The proof is obvious.
\end{proof}

The following theorem is analogous to Theorem \ref{theorem16*,p2} for IVF.

\begin{theorem}\label{theorem16**,p2}
	Let $E\subseteq (M,g)$ be star-shaped at $x_0 \in E$ and $f:E\rightarrow \mathbb{I}$ be an IVF with $f(x)= \langle f^c(x),~f^w(x)\rangle$. Let $\gamma(s);~s\in I, ~0\in I~ \text{and}~ \gamma(I)\subseteq E$, be a geodesic for which $\gamma(0)=x_0, ~ \&~ \dot{\gamma}(0)=X_{x_0}$ such that $(f^w \circ \gamma)(s)$ is non-decreasing for $s\in [0, \infty)$. If $f$ is cw-convex at $x_0$, then
	\begin{itemize}
		\item [(i)] for a fixed $X_{x_0} \in T_{x_0}(E)$, the function $Q:I\cap(0,\infty) \rightarrow \mathbb{R}$, defined by
		$$Q(s)=\frac{f(\gamma(s))\ominus_{gH}f(x_0)}{s},$$
		is non-decreasing;
		\item [(ii)] if $Df(x_0;~X_{x_0})$ exists, then 
		\begin{itemize}
			\item [(a)] $Df(x_0;~X_{x_0})= \lim\limits_{s\rightarrow 0^+}Q(s) = \langle Df^c(x_0;X_{x_0}),~Df^w(x_0;X_{x_0}) \rangle,$\\
			where,  $Df^w(x_0;X_{x_0})\geq 0$;
			\item [(b)] the IVF $h(X_{x_0}) = Df(x_0;~X_{x_0})$, defined on $T_{x_0}(E)$ is positively homogeneous.
		\end{itemize}
	\end{itemize}
\end{theorem}
\begin{proof}
	{\it (i)} Let $x\in E$ be arbitrary and $\gamma_{\2}(s), ~s\in [0,1]$ be a geodesic joining $x_0$ with $x$. From Definition \ref{definition**,p2} and Lemma \ref{lemma4**,p2}, we have
	$$g^c(s) := (f^c\circ \gamma_{\2})(s) ~~~\text{and}~~~ g^w(s) := (f^w\circ \gamma_{\2})(s), ~~~s\in[0,1],$$
	are convex at $s=0$.
	Let $B:=\{(s_1,s_2)\in (0,1] \times (0,1] : s_1 <s_2\}$. Then for any $(s_1,s_2)\in B$, from Lemma \ref{lemmaslope,p2}, we have
	\begin{align}
		\frac{g^c(s_1)-g^c(0)}{s_1} &\leq \frac{g^c(s_2)-g^c(0)}{s_2}\notag\\
		\& ~~~ \frac{g^w(s_1)-g^w(0)}{s_1} &\leq \frac{g^w(s_2)-g^w(0)}{s_2},\notag\\
		\Rightarrow~~~ \frac{f^c(\gamma_{\2}(s_1))-f^c(x_0)}{s_1} &\leq \frac{f^c(\gamma_{\2}(s_2))-f^c(x_0)}{s_2} \label{eq16,p2}\\
		\&~~~ \frac{f^w(\gamma_{\2}(s_1))-f^w(x_0)}{s_1} &\leq \frac{f^w(\gamma_{\2}(s_2))-f^w(x_0)}{s_2},~~~~\forall~(s_1,s_2)\in B. \label{eq17,p2}
		\intertext{Let $\displaystyle  T=\left\{(s_1,s_2)\in B:\frac{f^c(\gamma_{\2}(s_1))-f^c(x_0)}{s_1} = \frac{f^c(\gamma_{\2}(s_2))-f^c(x_0)}{s_2} \right\}$, then from (\ref{eq16,p2}) and (\ref{eq17,p2}), we can deduce that} 	
		\frac{f^c(\gamma_{\2}(s_1))-f^c(x_0)}{s_1} &< \frac{f^c(\gamma_{\2}(s_2))-f^c(x_0)}{s_2},~~~ \forall ~~ (s_1,s_2)\in B\setminus T\label{eq18,p2}\\
		\&~~~ \frac{f^w(\gamma_{\2}(s_1))-f^w(x_0)}{s_1} &\leq \frac{f^w(\gamma_{\2}(s_2))-f^w(x_0)}{s_2},~~~~\forall~(s_1,s_2)\in T.\label{eq19,p2}\\
		\intertext{As $(f^w \circ \gamma_{\2})(s)$ is non-decreasing for $s>0$, (\ref{eq19,p2}) becomes}
		\frac{|f^w(\gamma_{\2}(s_1))-f^w(x_0)|}{s_1} &\leq \frac{|f^w(\gamma_{\2}(s_2))-f^w(x_0)|}{s_2},~~~~\forall~(s_1,s_2)\in T.\label{eq20,p2}\\
		\intertext{From Lemma \ref{lemma13,p2} and order relation (\ref{orderrelation,p2}), together with (\ref{eq18,p2}) and (\ref{eq20,p2}), we have}
		\frac{f(\gamma_{\2}(s_1))\ominus_{gH}f(x_0)}{s_1} &\leq^{\text{min}} \frac{f(\gamma_{\2}(s_2))\ominus_{gH}f(x_0)}{s_2}.\notag
	\end{align}
	Since $x\in E$ is arbitrary, so for any $\gamma(s);~s\in I\cap(0,\infty), ~0\in I~ \&~ \gamma(I\cap(0,\infty))\subseteq E$, emanating from $\gamma(0)=x_0,$ with $\dot{\gamma}(0)=X_{x_0}$, we have
	\begin{align*}
		\frac{f(\gamma(s_1))\ominus_{gH}f(x_0)}{s_1} &\leq^{\text{min}} \frac{f(\gamma(s_2))\ominus_{gH}f(x_0)}{s_2}, ~~~ 0<s_1<s_2, ~~ s_1,s_2 \in I\cap (0, \infty), \notag\\
		\text{i.e.,}~~~~~~ Q(s_1) &\leq^{\text{min}} Q(s_2).\notag
	\end{align*}
	Since as $s_1,s_2\in I\cap(0,\infty)$  is arbitrary, we conclude that $Q(s)$ is non-decreasing.
	\vspace{0.30cm}
	
	\noindent {\it (ii)(a)} The proof follows directly from the Definition \ref{Definition4.2,p2} and Lemma \ref{lemma22,p2}.
	\vspace{0.05cm}
	
	\noindent {\it (ii)(b)} By hypothesis, we have 
	\begin{align*}
		h(X_{x_0}) &= Df(x_0;X_{x_0})\\
		&=\langle Df^c(x_0;X_{x_0}); |Df^w(x_0;X_{x_0})| \rangle.\\
		\intertext{So, for any $X_{x_0} \in T_{x_0}(E)$ and any $\lambda\geq 0,$ we have for sufficiently small $s \in I \cap (0, \infty)$ that}
		h(\lambda X_{x_0}) &= \langle Df^c(x_0;\lambda X_{x_0}); |Df^w(x_0;\lambda X_{x_0})| \rangle,\\
		\intertext{which by Theorem \ref{theorem16*,p2}, gives}
		h(\lambda X_{x_0})  &= \langle \lambda Df^c(x_0;X_{x_0});\lambda |Df^w(x_0;X_{x_0})| \rangle.\\
		\intertext{This together with Lemma \ref{lemma4,p2} yields}
		h(\lambda X_{x_0}) &= \lambda \langle Df^c(x_0;X_{x_0}); |Df^w(x_0;X_{x_0})| \rangle,\\
		h(\lambda X_{x_0}) &=\lambda h(X_{x_0}).
	\end{align*}
\end{proof}
\begin{remark}
	\rm Similar to Part $(ii) (a)$ in Theorem \ref{theorem16*,p2}, one might anticipate that Part $(ii) (a)$ in Theorem \ref{theorem16**,p2} would be:
	$$Df(x_0,X_{x_0}) = \underset{s}{\inf} ~Q(s).$$
	However, this equation is generally invalid because the order relation $\leq^{\text{min}}$ is not complete. The issue of the order incompleteness of $\leq^{\text{min}}$ has already been addressed in Remark \ref{remark2,p2}. The following example illustrates our assertion.
\end{remark}

\begin{example}
	\rm	Define $f: \mathbb{R} \rightarrow \mathbb{I}$ as follows:
	$$f(x) = \langle x^2, 1\rangle.$$
	Clearly, $f$ satisfies every hypothesis of the Theorem \ref{theorem16**,p2}. Choosing $x_0=0,~X_{x_0} = 1$ and $\gamma(s) = s$. We have
	$$Q(s) = \langle s, 0\rangle, ~~ s>0.$$
	Clearly, 
	$$Df(x_0,X_{x_0}) = \lim\limits_{s \rightarrow 0^+} ~Q(s) = \langle0,0\rangle.$$
	It is worth noting that $\langle0,0\rangle \leq^\text{min} Q(s)$ for any $s>0$. Hence, $\langle0,0\rangle$ is a lower bound of $Q(s)$. But the infimum of $Q(s)$ over $s>0$ doesn't exist.
\end{example}

In view of Lemma \ref{lemma22,p2}, we put forward the following theorem which gives a necessary condition for a gH-directionally differentiable cw-convex IVF.

\begin{theorem} \label{theorem23,p2}
	Let $E\subseteq (M,g)$ be star-shaped at $x_0 \in E$ and $f:E\rightarrow \mathbb{I}$ be an IVF with $f(x)= \langle f^c(x),~f^w(x)\rangle$. Let $\gamma(s);~s\in I, ~0\in I~ \text{and}~ \gamma(I)\subseteq E$, be any geodesic for which $\gamma(0)=x_0, ~ \&~ \dot{\gamma}(0)=X_{x_0}\in T_{x_0}(E)$ such that $(f^w \circ \gamma)(s)$ is non-decreasing for $s\in [0, \infty)$. Suppose that $f$ is gH-directionally differentiable at $x_0$, then
	\begin{enumerate}
		\item [i)] if $f$ is cw-convex  at $x_0$, then
		\begin{equation}\label{eq26,p2}
			Df(x_0;X_{x_0}) \leq^{\text{min}} f(x)\ominus_{gH}f(x_0), ~~ \forall ~ x\in E ~~\text{and}~~ \forall ~\gamma_{\2}\in \Gamma_0,
		\end{equation}
	where $\Gamma_0$ is the collection of all geodesics joining $x_0$ and $x$ such that $\gamma_{\2}(0)=x_0$ and $\dot{\gamma}_{\2}(0)=X_{x_0}\in T_{x_0}(E).$
	\item [ii)] if $f$ is strictly cw-convex  at $x_0$, then
	$$Df(x_0;X_{x_0}) <^{\text{min}} f(x)\ominus_{gH}f(x_0), ~~ \forall ~ x\in E,~~x\neq x_0, ~~\text{and}~~ \forall ~\gamma_{\2}\in \Gamma_0,$$
	with $\gamma_{\2}(0)=x_0$ and $\dot{\gamma}_{\2}(0)=X_{x_0}\in T_{x_0}(E).$
	\end{enumerate}
\end{theorem}

\begin{proof}
	From Lemma \ref{lemma22,p2}, we have
	\begin{equation} \label{eq27,p2}
		Df(x;X_x)~ = ~\langle Df^c(x;X_x),~|Df^w(x;X_x)|\rangle.
	\end{equation}
From the convexity of $f^c$ and $f^w$, we have
\begin{align}
	Df^c(x_0;X_{x_0}) &\leq f^c(x) - f^c(x_0), ~~ \forall~x \in E,\label{eq28,p2}\\
   	Df^c(x_0;X_{x_0}) &\leq f^w(x) - f^w(x_0), ~~ \forall~x \in E.\label{eq29,p2}
\end{align}
    Let $T:= \{x\in E:Df^c(x_0;X_{x_0}) \leq f^c(x) - f^c(x_0)\},$ then from (\ref{eq28,p2}) and (\ref{eq29,p2}), we can deduce that
\begin{align}
    	Df^c(x_0;X_{x_0}) &< f^c(x) - f^c(x_0), ~~ \forall~x \in E\setminus T,\label{eq30,p2}\\
    	Df^w(x_0;X_{x_0}) &\leq f^w(x) - f^w(x_0), ~~ \forall~x \in T.\label{eq31,p2}
\end{align}
Also, by the hypothesis, we have $(f^w\circ \gamma_{\2})(s)$ is non-decreasing for $s \geq 0$. So,
$$Df^w(x_0;X_{x_0}) = \lim\limits_{s\rightarrow0^+} \frac{f^w(\gamma_{\2}(s))-f^w(\gamma_{\2}(0))}{s}\geq 0.$$
This together with (\ref{eq31,p2}), yields
\begin{equation} \label{eq32,p2}
	|Df^w(x_0;X_{x_0})| \leq |f^w(x) - f^w(x_0)|, ~~ \forall~x \in T.
\end{equation}
From order relations (\ref{orderrelation,p2}), together with (\ref{eq30,p2}) and (\ref{eq32,p2}), for any $x\in E$ we conclude that
$$\langle Df^c(x_0;X_{x_0});~|Df^w(x_0;X_{x_0})| \rangle \leq^{\text{min}} \langle f^c(x) - f^c(x_0);~ |f^w(x) - f^w(x_0)| \rangle,$$
which from Lemma \ref{lemma13,p2} and (\ref{eq27,p2}) yields
$$Df(x_0;X_{x_0}) \leq^{\text{min}} f(x)\ominus_{gH}f(x_0), ~~ \forall ~ x\in E.$$
\end{proof}


The following example supports the fact that the removal of the non-decreasing condition of $(f^w\circ \gamma)(s)$ makes Theorem \ref{theorem23,p2} invalid, in general.

\begin{example} \label{example24,p2} 
		{\rm \bfseries Manifold (or Cone) of symmetric positive definite matrices: }\\
	\rm 
	The collection, $S^n_{++}$, of $n \times n$ symmetric positive definite matrices with real entries forms a Hadamard manifold with Riemannian metric:
	$$g_P(X,Y) = Tr(P^{-1}XP^{-1}Y), ~~~ \forall~ P \in S^n_{++}, ~~~ X,Y \in T_P(S^n_{++}).$$
	The minimal geodesic joining $P, Q \in S^n_{++}$ is given by 
	$$\gamma(s) = P^{\frac{1}{2}}(P^{-\frac{1}{2}}QP^{-\frac{1}{2}})^sP^{\frac{1}{2}}, ~~~ \forall ~ s \in [0,1].$$
	For more details, one can refer to \cite{bacak,serge}.\\
For $n=2$, we define IVF $f:S^2_{++} \rightarrow \mathbb{I}$ as 
$$f(x)=\langle \ln(\det(x)),~(\ln(\det(x)))^2 \rangle.$$
For any $x,y \in S^2_{++}$, the geodesic joining them is given by
$$\gamma_{\1}(s) = x^{\frac{1}{2}}\left(x^{-\frac{1}{2}}yx^{-\frac{1}{2}}\right)^sx^{\frac{1}{2}}, ~~~ \forall ~ s \in [0,1].$$
One can show that both $f^c$ and $f^w$ are convex on $S^2_{++}$,  and hence $f$ is cw-convex on $S^2_{++}$. Also, one can show that $f$ is gH-directionally differentiable on $S^2_{++}$ with gH-directional derivative at any $x$ in any direction $X_x=\dot{\gamma}_{\1}(0)$ as
$$Df(x;X_x) = \left\langle \ln(\det(y))-\ln(\det(x)),~ |2\ln(\det(x))[\ln(\det(y))-\ln(\det(x))]| \right\rangle.$$
We now choose $x=\frac{1}{2}I_2$ and $y=I_2$, then 
$$Df^c(x;X_x)\approx 1.38,~~~ \text{and}~~~ Df^w(x;X_x)\approx -3.81.$$
Since as $Df^w(x;X_x)<0$, we have $(f^w\circ \gamma_{\1})(s), ~s\in[0,1]$ is decreasing.
\text{Also}, we have the following:
\begin{align*}
	Df(x;X_x) &= \langle 1.38,~ 3.81 \rangle.\\
\text{and}~~~ f(y) \ominus_{gH} f(x) &=\langle 1.38,~ 1.90 \rangle.
\end{align*}
Hence, by order relation (\ref{orderrelation,p2}), we have 
$$f(y) \ominus_{gH} f(x) \leq^{\text{min}} Df(x;X_x),$$
which clearly contradicts with (\ref{eq26,p2}).
\end{example}

\begin{remark}\rm 
	One should note that the condition of being non-decreasing in $[0, \infty)$ for $(f^w\circ \gamma)(t)$ is not merely for the existence  of directional derivative of $f^w$ as said in Lemma \ref{lemma22,p2} but it also assures that $Df^w(x_0;X_{x_0})\geq 0$ which results in the validity of relation (\ref{eq26,p2}) and (\ref{eq27,p2}).
\end{remark}

In general, it is not possible to invert Theorem \ref{theorem23,p2}, as illustrated in the following example.

\begin{example}\rm 
Consider the Riemannian manifold \( \mathit{M} \) as defined in Example \ref{exg4.2} and choose $x_0 = e^{i\frac{\pi}{2}}=i$ then The geodesic segment joining $x=i$ with any $x=e^{i\theta} \in M$ is given by
$$\gamma_{\2}(s) = e^{i((1-s)\frac{\pi}{2}+s\theta)}, ~~~ s\in [0,1],$$	
with $\dot{\gamma}_{\2}=i(\theta-\frac{\pi}{2}).$
\vspace{0.15cm}

Let $f:M\rightarrow \mathbb{I}$, be defined as follows:
$$f(x)=f(e^{i\theta}) = \left \langle \theta^2,~ln\left((\theta-\frac{\pi}{2})^2+e\right) \right \rangle.$$
Here, $f$ is gH-directionally differentiable at $x_0$ with gH-directional derivative as
$$Df(x_0;X_{x_0}) =\left\langle \pi\left(\theta - \frac{\pi}{2}\right),~ 0 \right\rangle.$$
Here, $f^c$ is strictly convex at $x_0=i$. So, for any $x=e^{i\theta}$ and the geodesic segment $\gamma_{\2} \in \Gamma_0$, we have
$$Df^c(x_0;X_{x_0}) <  f^c(x) - f^c(x_0),~~ \forall~x\in M,~~x\neq x_0, ~~\gamma_{\2}\in \Gamma_0.$$
From order relations (\ref{orderrelation,p2}) and Lemma \ref{lemma13,p2}, we have
$$Df(x_0;X_{x_0}) \leq^{\text{min}}  f(x) \ominus_{gH} f(x_0),~~ \forall~x\in M,~~x\neq x_0, ~~\gamma_{\2}\in \Gamma_0.$$
Now, for any $\gamma_{\2}(s) = e^{i((1-s)\frac{\pi}{2}+s\theta)} \in \Gamma_0$, we have
\begin{align*}
	(f^w \circ \gamma_{\2})(s)&=ln\left((\theta-\frac{\pi}{2})^2s^2 + e\right)\\
	\text{and}~~~ (f^w \circ \gamma_{\2})'(s) &= \frac{2(\theta - \frac{\pi}{2})^2s}{(\theta -\frac{\pi}{2})^2s^2 +e} \geq 0, ~~~ \forall ~ s \in [0,1].
\end{align*}
This shows $f^w \circ \gamma_{\2}$ is non-decreasing for $s \geq 0$ for any $\gamma_{\2} \in \Gamma_0$.
\vspace{0.15cm}

But $f^w$ is not convex at $x_0=i$. For this, one can check the following
$$f^w\left(\gamma_{\2}(\frac{1}{2})\right) > \frac{1}{2}f(e^{i\frac{\pi}{2}}) + \frac{1}{2}f(e^{i\frac{3\pi}{2}}).$$
This shows that $f$ is not cw-convex at $x_0=i$.
\end{example}

\section{Conclusion}
In this article, we have established that gH-directional differentiability of an IVF can be expressed in terms of its center function, half-width function, and endpoint functions, but only for a specific category of functions. We have employed a total order relation for the collection of closed and bounded intervals of $\mathbb{R}$ throughout our analysis. Our research represents an initial step towards the development of KKT-type optimality conditions for optimization programming problems based on gH-directional differentiability in the context of a total order relation. These findings can be extended to fuzzy environments, and they have potential applications in verifying the convexity of an IVF, as well as in creating numerical methods for solving optimization programming problems.

\section{Appendix A}
Here we present the counter examples of Remark 2.1(i) in \cite{chen} and Lemma 3.2 in \cite{chen}.
   
   The following example assures that the Remark 2.1(i) in \cite{chen} is false.
   
   \begin{example} \label{example6.1,p2} \rm
   	Let $M=\mathbb{R}_{++}=\{x\in \mathbb{R}: x >0\}$ be a manifold with Riemannian metric given by
   	$$\langle v,w \rangle_x = \frac{1}{x^2}vw, ~~~ \forall~v,w \in T_x(M)=\mathbb{R}.$$
   	It is known that $M$ is a Hadamard manifold. For any $x \in M$ and $v\in T_x(M)$, the geodesic $\gamma:\mathbb{R} \rightarrow M$ emanating from $x=\gamma(0)$ in the direction $\gamma'(0)=v$ is given by
   	$$\gamma(s)=\text{exp}_x(sv)=xe^{(\frac{v}{x})s}.$$
   	 	Define $f:M \rightarrow \mathbb{R}$, as following:
   	$$f(x)=\begin{cases}
   		1, &x\leq 1;\\
   		-\text{ln}(x), &x>1.
   	\end{cases}$$
    We choose $x=1$, then the geodesic segment $\gamma_{\1}(s),~ s \in [0,1]$, joining $x=1$ with any $y\in M$ is given by
   $$\gamma_{\1}(s)=y^s.$$
   So, 
\begin{align*}
	   (f\circ \gamma_{\1})(s) &= \begin{cases}
   	1, &y\leq 1;\\
   	-s \cdot \text{ln}(y), &y>1.
   \end{cases}\\
   	\intertext{Also,} \\
   	(1-s)f(x) + s f(y) &=\begin{cases}
   		1, &y\leq 1;\\
   		(1-s) - s\cdot\text{ln}(y), &y>1.
   	\end{cases}
\end{align*}
It is easy to see that for $x=1$, and any $y\in M$, we have the following:
$$(f\circ \gamma_{\1})(s)\leq (1-s)f(x) + s f(y), ~~ \forall~ s \in [0,1].$$
Hence, we conclude that $f$ is $\text{convex}^{\dagger}$\let\thefootnote\relax\footnotetext{$\dagger$The definition of convexity is given by Definition 2.4 in \cite{chen} as following:\\ Let $M$ be a Hadamard manifold and $D\subseteq M$ be a nonempty open geodesic convex set. A function $f: D \rightarrow \mathbb{R}$ is said to be convex at $x \in D$, if for any $y \in D$,
	$$f(\gamma(s))\leq s f(y) + (1-s)f(x), ~~~ \text{for all} ~ s \in[0,1].$$
	Where $\gamma : [0,1] \rightarrow D$ is a geodesic with $\gamma(0) = x$ and $\gamma(1) = y$.} at $x=1$.
   	
   	Also, directional derivative*\let\thefootnote\relax\footnotetext{*Directional derivative is given by Definition 2.3 in \cite{chen} as following:\\ Let $M$ be a Hadamard manifold, let $D \subseteq M$ be a nonempty open set, and let $f:D \rightarrow \mathbb{R}$ be a function. We say that $f$ is directionally differentiable at a point $x \in D$ in the direction of $v \in T_x(M)$ if the limit 
   			$$f'(x;v) = \lim\limits_{s \rightarrow 0^+} \frac{f(\exp_x ~ s v)-f(x)}{s}$$
   			exists, where $f'(x;v)$ is called the directional derivative of $f$ at $x$ in the direction of $v \in T_x(M).$  If $f$ is directionally differentiable at $x$ in every direction $v \in T_x(M)$, we say that $f$ is directionally differentiable at $x$.}
 of $f(x)$ at $x=1$ in any positive direction $v \in \mathbb{R}$ doesn't exist finitely. In particular, choose $y=2$. The geodesic segment $\gamma_{\1}(s)=2^s, ~t \in [0,1],$ joins $x=\gamma_{\1}(0)=1$ with $y=\gamma_{\1}(1)=2$. The directional derivative of $f$ at $x=\gamma_{\1}(0)=1$ in the direction $\gamma'_{\1}(0)=\text{ln}(2)$ is
   	$$f'(1;\text{ln}(2))=\lim\limits_{t\rightarrow0^+} \frac{f(2^s)-f(1)}{s}= -\text{ln}(2)-\lim\limits_{s\rightarrow0^+} \frac{1}{s}=-\infty,$$
   	which is not finite. Hence, $f$ is not directionally differentiable at $x=1.$
   \end{example}

%
%
%
%

Next, the following example demonstrates that the converse statement of Lemma 3.2 in \cite{chen} does not hold in general.

\begin{example}\rm 
	We consider the same manifold $S^n_{++}$ as presented in Example \ref{example24,p2}. Let $n=2$, we have $S^2_{++}$ as the Hadamard manifold. For $P=I_2$ and $Q=2I_2$, where $I_2$ is a $2 \times 2$ identity matrix, the geodesic segment joining $P$ and $Q$ is given by
	$$ \gamma_{pq}(s) = I_2^{\frac{1}{2}}(I_2^{-\frac{1}{2}}(2I_2)I_2^{-\frac{1}{2}})^sI_2^{\frac{1}{2}}, ~~~ \forall ~ s \in [0,1],$$ 
	$$ i.e., ~~~~~ \gamma_{pq}(s) = 2^sI_2, ~~~ \forall ~ s \in [0,1].~~~~~~~~~~~~~~~~~~~~~~~~~~~~~$$
	Here, $\gamma_{pq}(0) = I_2 =P$ and it easy to see $\dot{\gamma}_{pq}(0)= \ln(2)I_2 = V \in T_P(S^2_{++})$\\
	Define $f: S^2_{++} \rightarrow \mathbb{I}$, as follows
	\begin{align*}
		f(x)&= [f^l(x), f^U(x)]\\
		&= \begin{cases}
			[1,2], & x \in S^2_{++} \backslash \gamma_{pq}(s),~ t\in [0,1],\\
			[0,~ \det(x)], & x \in \gamma_{pq}(s),~ s \in [0,1]\cap \mathbb{Q},\\
			[\ln(\det(x)),~ (\ln(\det(x))^2 +1], & x \in \gamma_{pq}(s),~s \in  [0,1]\cap \mathbb{R} \backslash \mathbb{Q}.
		\end{cases}
	\end{align*}
	The lower and upper end point functions are, respectively, given by
	$$f^l(x)= \begin{cases}
		1, & x \in S^2_{++} \backslash \gamma_{pq}(s),~ s\in [0,1],\\
		0, & x \in \gamma_{pq}(s),~ s \in [0,1]\cap\mathbb{Q},\\
		\ln(\det(x)), & x \in \gamma_{pq}(s), ~ s \in [0,1]\cap\mathbb{R}\backslash\mathbb{Q}.
	\end{cases}$$ 
	$$f^u(x)= \begin{cases}
		2, & x \in S^2_{++} \backslash \gamma_{pq}(s), ~s\in[0,1],\\
		\det(x), & x \in \gamma_{pq}(s),~s \in [0,1]\cap\mathbb{Q},\\
		(\ln(\det(x)))^2 +1, & x \in \gamma_{pq}(s),~ s \in [0,1]\cap\mathbb{R}\backslash\mathbb{Q}.
	\end{cases}$$
	The limit of the slope function of secants to $f^l$ at $P=\gamma_{pq}(0)=I_2$ in the direction $V=\dot{\gamma}_{pq}(0)=\ln(2)I_2 \in T_P(S^2_{++})$ is given by
	\begin{align*}
		(f^l)'(P;V)=\lim\limits_{s\rightarrow0^+} \frac{f^l(\gamma(s))-f^l(P)}{s}
		&= \lim\limits_{s\rightarrow0^+} \frac{f^l(2^sI_2)-f^l(\gamma(0))}{s}\\
		&= \lim\limits_{s\rightarrow0^+} \begin{cases}
			0, & s \in [0,1]\cap\mathbb{Q},\\
			\frac{\ln(\det(2^sI_2))}{s}, & s \in [0,1]\cap\mathbb{R}\backslash\mathbb{Q},
		\end{cases}\\
		&= \lim\limits_{s\rightarrow0^+} \begin{cases}
			0, & s \in [0,1]\cap\mathbb{Q},\\
			\frac{\ln(2^s)^2)}{s}, & s \in [0,1]\cap\mathbb{R}\backslash\mathbb{Q},
		\end{cases}\\
		&= \begin{cases}
			0, & s \in [0,1]\cap\mathbb{Q},\\
			\ln(4), & s \in [0,1]\cap\mathbb{R}\backslash\mathbb{Q}.
		\end{cases}
	\end{align*}
	
	As the limit is not same for all values of $s > 0$, we conclude that $f^l$ is not directional differentiable at $I_2$ in the direction $V$.
	
	Similarly, the limit of the slope function of secants to $f^u$ at $P=I_2 \in E$ in the direction $V=\ln(2)I_2 \in T_P(S^2_{++})$, turns out to be
	\begin{align*}
		(f^u)'(P;V) = \begin{cases}
			\ln(4), & s \in [0,1]\cap\mathbb{Q},\\
			0, & s \in [0,1]\cap\mathbb{R}\backslash\mathbb{Q},
		\end{cases}
	\end{align*}
	which shows, $f^u$ is also not directional differentiable at $I_2$ in the direction $V$.
	
	But, the gH-directional derivative of $f$ at $P=I_2$ in the direction $V=\ln(2)I_2 \in T_P(E)$ exists, as is evident from the following
	\begin{align*}
		\frac{f(\gamma(s) \ominus_g f(\gamma(P)))}{s} 
		&= \frac{f(\gamma(s))\ominus_g f(\gamma(0))}{s}\\
		&= \frac{1}{s}\big[\text{min}\{f^l(\gamma(s))-f^l(I), f^u(\gamma(s))-f^u(I)\},\\ &~~~~~~~~~~~~~~\text{max}\{f^l(\gamma(s))-f^l(I), f^u(\gamma(s))-f^u(I)\}\big]\\
		&=\bigg[\text{min}\bigg\{\frac{f^l(\gamma(s))-f^l(\gamma(0))}{s}, \frac{f^u(\gamma(s))-f^u(\gamma(0))}{s}\bigg\}, \\
		& ~~~~~~\text{max}\bigg\{\frac{f^l(\gamma(s))-f^l(\gamma(0))}{s}, \frac{f^u(\gamma(s))-f^u(\gamma(0))}{s} \bigg\} \bigg].
	\end{align*}
	For sufficiently small $s > 0$ (or $0<s<\frac{1}{\ln(4)}$), we have 
	$$\frac{f(\gamma(s) \ominus_g f(\gamma(P)))}{s}= \begin{cases}
		\big[0, \frac{4^s - 1}{s}\big], & s \in [0,1]\cap\mathbb{Q},\\
		[s(\ln(4))^2,~ \ln(4)], & s \in [0,1]\cap\mathbb{R}\backslash\mathbb{Q},
	\end{cases}$$
		$$\implies ~ f'(P;V)= \lim\limits_{s \rightarrow 0^+} \frac{f(\gamma(s) \ominus_g f(\gamma(P)))}{s} = [0, \ln(4)].$$
	Thus, the necessary part of the Lemma 3.2 in \cite{chen} is not true. 
\end{example}


\end{document}